 \documentclass[10pt]{article}

 \usepackage{latexsym,enumerate}
 \usepackage{mathrsfs}
 \usepackage{amsfonts}
 \usepackage{amssymb,amsmath,amsthm}
 \usepackage{epsfig}       
 \usepackage{subfigure}    
 \usepackage{graphicx}
 \usepackage{float}
 \usepackage{multirow}
 \usepackage{array}
 \usepackage{color}
 \usepackage{url}
 \DeclareMathOperator*{\argmin}{argmin}

 \numberwithin{equation}{section}

 \def \R{{\mathbb{R}}}

 \def \dom{{{\rm dom}\,}}
 
 \def\argmin{\mathop{\rm arg\,min}}
 \def\Argmin{\mathop{\rm Arg\,min}}

 \newtheorem{proposition}{Proposition}[section]
 \newtheorem{definition}{Definition}[section]
 \newtheorem{lemma}{Lemma}[section]
 
 \newtheorem{theorem}{Theorem}[section]
 \newtheorem{remark}{Remark}[section]
 \newtheorem{assumption}{Assumption}[section]

 \theoremstyle{definition}
 \newtheorem{example}{Example}[section]

\title{A Proximal Difference-of-convex Algorithm with Extrapolation}

\author{
Bo Wen \thanks{Department of Mathematics, Harbin Institute of Technology, Harbin, P.R. China. Current address: Department of Applied Mathematics, The Hong Kong Polytechnic University, Hong Kong, P.R. China.
E-mail: {bo.wen@connect.polyu.hk}.}
\and Xiaojun Chen \thanks{Department of Applied Mathematics, The Hong Kong Polytechnic University, Hong Kong, P.R. China. This author's work was supported in part by Hong Kong Research Grants Council PolyU153000/15p. E-mail: {maxjchen@polyu.edu.hk}.}
\and Ting Kei Pong \thanks{Department of Applied Mathematics, The Hong Kong Polytechnic University, Hong Kong, P.R. China.
This author's work was supported in part by Hong Kong Research Grants Council PolyU253008/15p. E-mail: {tk.pong@polyu.edu.hk}.}
}

\date{June 17, 2017}

\begin{document}

\maketitle

\begin{abstract}

We consider a class of difference-of-convex (DC) optimization problems whose objective is level-bounded and is the sum of a smooth convex function with Lipschitz gradient, a proper closed convex function and a continuous concave function. While this kind of problems can be solved by the classical difference-of-convex algorithm (DCA) \cite{PT1997}, the difficulty of the subproblems of this algorithm depends heavily on the choice of DC decomposition. Simpler subproblems can be obtained by using a specific DC decomposition described in \cite{PT1998}. This decomposition has been proposed in numerous work such as \cite{GTT2015}, and we refer to the resulting DCA as the proximal DCA. Although the subproblems are simpler, the proximal DCA is the same as the proximal gradient algorithm when the concave part of the objective is void, and hence is potentially slow in practice. In this paper, motivated by the extrapolation techniques for accelerating the proximal gradient algorithm in the convex settings, we consider a proximal difference-of-convex algorithm with extrapolation to possibly accelerate the proximal DCA. We show that any cluster point of the sequence generated by our algorithm is a stationary point of the DC optimization problem for a fairly general choice of extrapolation parameters: in particular, the parameters can be chosen as in FISTA with fixed restart \cite{DC2015}. In addition, by assuming the Kurdyka-{\L}ojasiewicz property of the objective and the differentiability of the concave part, we establish global convergence of the sequence generated by our algorithm and analyze its convergence rate. Our numerical experiments on two difference-of-convex regularized least squares models show that our algorithm usually outperforms the proximal DCA and the general iterative shrinkage and thresholding algorithm proposed in \cite{GZLHY2013}.

\vspace{5mm}

\noindent {\bf Keywords:} difference-of-convex problems, nonconvex, nonsmooth, extrapolation, Kurdyka-{\L}ojasiewicz inequality \hspace{2mm}
\vspace{2mm}

\end{abstract}

\section{Introduction}

Difference-of-convex (DC) optimization problems are problems whose objective can be written as the difference of a proper closed convex function and a continuous convex function. They arise in various applications such as digital communication system \cite{ASP2014}, assignment and power allocation \cite{SRL2014} and compressed sensing \cite{YLHX2015}; we refer the readers to Sections~7.6 to 7.8 of the recent monograph \cite{T2016} for more applications of DC optimization problems.

A classical algorithm for solving DC optimization problems is the so-called DC algorithm (DCA), which was proposed by Tao and An \cite{PT1997}; see also \cite{BB2016,GTT2015,TP2005,TPH2012,TPL1999} for more recent developments.\footnote{We would also like to point to the article ``DC programming and DCA" on the person webpage of Le Thi Hoai An:
\url{http://www.lita.univ-lorraine.fr/~lethi/index.php/en/research/dc-programming-and-dca.html}} In each iteration, this algorithm replaces the concave part of the objective by a linear majorant and solves the resulting convex optimization problem. The difficulty of the subproblems involved relies heavily on the choice of DC decomposition of the objective function. When the objective can be written as the sum of a smooth convex function with Lipschitz gradient, a proper closed convex function and a continuous concave function, simpler subproblems can be obtained by using a specific DC decomposition described in \cite[Eq.~16]{PT1998}. This idea appears in numerous work and is also recently adopted in \cite{GTT2015}, where they proposed the so-called proximal DCA.\footnote{This algorithm was called ``the proximal difference-of-convex decomposition algorithm" in \cite{GTT2015}. As noted in \cite{GTT2015}, their algorithm is the DCA applied to a specific DC decomposition.} This algorithm not only majorizes the concave part in the objective by a linear majorant in each iteration, but also majorizes the smooth convex part by a quadratic majorant. When the proximal mapping of the proper closed convex function is easy to compute, the subproblems of the proximal DCA can be solved efficiently. However, this algorithm may take a lot of iterations: indeed, when the concave part of the objective is void, the proximal DCA reduces to the proximal gradient algorithm for convex optimization problems, which can be slow in practice \cite[Section~5]{DC2015}.

It is then tempting to incorporate techniques to possibly accelerate the proximal DCA while not significantly increasing the computational cost per iteration. One such technique is to perform extrapolation. More precisely, this means adding {\em momentum} terms that involve previous iterates for updating the current iterate. Such technique has been adopted for convex optimization problems, dating back to Polyak's heavy ball method \cite{P1964}. More recent examples of such techniques are Nesterov's extrapolation techniques \cite{N1983,N2004,N2007,N20071} which have been extensively used for accelerating the proximal gradient algorithm and its variants for convex optimization problems. One representative algorithm that incorporates these techniques is the fast iterative shrinkage-thresholding algorithm (FISTA) \cite{BT2009,N2007}. It is known that the function values generated by FISTA converges at a rate of $O(1/k^2)$, which is faster than the $O(1/k)$ convergence rate of the proximal gradient algorithm. We refer the readers to \cite{BCG11,DC2015} for more examples of such algorithms.

In view of the success of extrapolation techniques in accelerating the proximal gradient algorithm for convex optimization problems, and noting that the proximal gradient algorithm and the proximal DCA are the same when applied to convex problems, in this paper, we incorporate extrapolation techniques to possibly accelerate the proximal DCA in the general DC settings.\footnote{It is also discussed at the end of the numerical section of \cite{GTT2015} that suitably incorporating extrapolation techniques into the proximal DCA can accelerate the algorithm empirically.} We call our algorithm the proximal DCA with extrapolation (${\rm pDCA}_e$). We prove that, for a fairly general choice of extrapolation parameters, if the objective is level-bounded, then any cluster point of the sequence generated by our algorithm is a stationary point of the DC optimization problem. The choice of parameters is general enough to cover those used in FISTA with fixed restart \cite{DC2015}. Additionally, by assuming that the objective is a level-bounded Kurdyka-{\L}ojasiewicz function (see, for example, \cite{ABRS2010}) and the concave part is differentiable with a locally Lipschitz gradient, we establish global convergence of the whole sequence generated by our algorithm. We also analyze the convergence rate based on the Kurdyka-{\L}ojasiewicz exponent. Finally, we perform numerical experiments on $\ell_{1-2}$ \cite{YLHX2015} and logarithmic \cite{CWB2008} regularized least squares problems. Our numerical experiments show that the ${\rm pDCA}_e$ usually outperforms the proximal DCA and the general iterative shrinkage and thresholding algorithm (GIST) proposed in \cite{GZLHY2013}.

The rest of this paper is organized as follows. In Section~\ref{sec2}, we introduce notation and discuss some preliminary materials. In Section~\ref{sec3}, we describe the DC optimization problem we study in this paper and present our algorithm ${\rm pDCA}_e$. The convergence of the sequence generated by the algorithm and the convergence rate are studied in Section~\ref{sec4}. Finally, we present numerical experiments in Section~\ref{sec5}.

\section{Notation and preliminaries}\label{sec2}
In this paper, we use $\R^n$ to denote the $n$-dimensional Euclidean space with inner product $\langle\cdot,\cdot\rangle$ and Euclidean norm $\|\cdot\|$, and use $\|\cdot\|_1$ and $\|\cdot\|_\infty$ to denote the $\ell_1$ norm and the $\ell_\infty$ norm, respectively. Given a matrix $A \in \R^{m\times n}$, the transpose of $A$ is denoted by $A^{T}$. Moreover, for a symmetric matrix $A \in \R^{n\times n}$, we use $\lambda_{\rm max}(A)$ and $\lambda_{\rm min}(A)$ to denote its largest and smallest eigenvalues, respectively. In addition, for a nonempty closed set $\mathcal{C} \subseteq \R^n$, we denote the distance from a point $x\in \R^n$ to $\mathcal{C}$ by $\mathrm{dist}(x,\mathcal{C}):=\inf_{y\in\mathcal{C}}\|x-y\|$.

For an extended-real-valued function $h: \R^n\to [-\infty,\infty]$, we denote its domain by $\dom h = \left\{x\in \R^{n}:\;  h(x) < \infty \right\} $. The function $h$ is said to be proper if it never equals $-\infty$ and $\dom h\neq \emptyset$. Moreover, a proper function is closed if it is lower semicontinuous. A proper closed function $h$ is said to be level-bounded if the lower level sets of $h$ (i.e., $\left\{x\in\R^{n}:\;  h(x)\leq r \right\}$ for any $r\in \R$) are bounded. Given a proper closed function $h:{{\mathbb{R}}^{n}}\to \mathbb{R}\cup \{\infty\}$, the (limiting) subdifferential of $h$ at $x\in \dom h$ is given by
\begin{equation}\label{subdiffh}
\partial h(x) = \left\{v\in\R^n: \exists~ x^t\stackrel{h}{\rightarrow} x, v^t\rightarrow v \ {\rm with}\ \liminf\limits_{y \rightarrow x^t} \frac{h(y)-h(x^t)-\langle v^t, y-x^t \rangle}{\|y-x^t\|}\ge 0\ \ {\rm for~each~} t \right\},
\end{equation}
where $z\stackrel{h}{\rightarrow}x$ means $z\to x$ and $h(z)\to h(x)$. We also write ${\rm dom}\,\partial h:= \{x\in \mathbb{R}^n:\; \partial h(x)\neq \emptyset\}$.
It is known that the above subdifferential reduces to the classical subdifferential in convex analysis when $h$ is convex, i.e.,
\begin{equation*}
\partial h(x)=\left\{v\in{\mathbb{R}}^{n}:\;  h(u)-h(x)-\langle v, u-x \rangle \geq 0,\ \forall  u\in {\mathbb{R}}^{n}\right\};
\end{equation*}
see, for example, \cite[Proposition~8.12]{RR1998}. In addition, if $h$ is continuously differentiable, then the subdifferential \eqref{subdiffh} reduces to the gradient of $h$ denoted by $\nabla h$. We also use $\nabla_i h$ to denote the partial gradient of $h$ with respect to $x_i$, the $i$-th component of $x$.


We next recall the Kurdyka-{\L}ojasiewicz (KL) property \cite{AB2009,ABRS2010,ABF2013,BDL2007}, which is satisfied by a wide variety of functions such as proper closed semialgebraic functions, and plays an important role in the convergence analysis of many first-order methods; see, for example, \cite{ABRS2010,ABF2013}.

\begin{definition}{\bf(KL property)}\label{KLF}
A proper closed function $h$ is said to satisfy the KL property at $\hat{x}\in \mathrm{dom}\, \partial h$ if there exist $a \in (0,\infty]$, a neighborhood $\mathcal{O}$ of $\hat{x}$, and a continuous concave function $\phi: [0,a)\rightarrow \R_{+}$ with $\phi(0)=0$ such that:
\begin{enumerate}[{\rm (i)}]
  \item $\phi$ is continuously differentiable on $(0,a)$ with $\phi'>0$;
  \item For any $x \in \mathcal{O}$ with $h(\hat{x})<h(x)<h(\hat{x})+a$, one has
\begin{equation}\label{KL_equation}
\phi'(h(x)-h(\hat{x}))\,\mathrm{dist}(0,\partial h(x))\ge 1.
\end{equation}
\end{enumerate}
A proper closed function $h$ satisfying the KL property at all points in ${\rm dom}\, \partial h$ is called a KL function.
\end{definition}
We also recall the following result proved in \cite[Lemma 6]{BST2014} concerning the uniformized KL property.
For notational simplicity, we use $\Xi_a$ to denote the set of all concave continuous functions $\phi: [0,a)\rightarrow \R_{+}$ that are continuously differentiable on $(0,a)$ with positive derivatives and satisfy $\phi(0)=0$.
\begin{lemma}{\bf(Uniformized KL property)}\label{UNKLF}
Suppose that $h$ is a proper closed function and let $\Gamma$ be a compact set. If $h$ is a constant on $\Gamma$ and satisfies the KL property at each point of $\Gamma$, then there exist $\epsilon, a>0$ and $\phi \in \Xi_a$ such that
\begin{equation*}
\phi'(h(x)-h(\hat{x}))\mathrm{dist}(0,\partial h(x))\ge 1
\end{equation*}
for any $\hat{x}\in \Gamma$ and any $x$ satisfying $\mathrm{dist}(x,\Gamma) < \epsilon$ and $h(\hat{x})<h(x)<h(\hat{x})+a$.
\end{lemma}

\section{Problem formulation and the proximal difference-of-convex algorithm with extrapolation}\label{sec3}

In this section, we describe the optimization problem we study in this paper and present our proximal difference-of-convex algorithm with extrapolation ($\rm pDCA_e$).

We focus on problems of the following form:
\begin{equation}\label{P1}
  v:= \min_{x\in \R^n} F(x) := \ f(x) + P(x),
\end{equation}
where $f$ is a smooth convex function with a Lipschitz continuous gradient whose Lipschitz continuity modulus is $L>0$, and
\[
P(x) = P_1(x) - P_2(x),
\]
with $P_1$ being a proper closed convex function and $P_2$ being a {\em continuous} convex function. We assume in addition that $F$ is level-bounded. This latter assumption implies that $v > -\infty$ and that the set of global minimizers of \eqref{P1} is nonempty. Problem~\eqref{P1} arises in applications such as compressed sensing, where $f$ is typically the data fitting term such as the least squares loss function, and $P$ is a nonsmooth regularizer for inducing desirable structures in the solution. We refer the readers to \cite{APX2016, BC2016, FL2001,YLHX2015,Z2010,ZX2016} for concrete examples.

It is clear that problem \eqref{P1} is a DC optimization problem and can be solved by the renowned DCA. However, as noted in the introduction, the difficulty of the subproblems involved in the DCA depends on the DC decomposition used. Indeed, when decomposing $F$ naturally as the difference of $f+P_1$ and $P_2$, the subproblems of the corresponding DCA take the following form:
\begin{equation}\label{DCAsub}
x^{t+1} \in \Argmin_{x\in \R^n}\left\{f(x) + P_1(x) - \langle\xi^t,x\rangle\right\},
\end{equation}
where $\xi^t\in \partial P_2(x^t)$. Although these problems are convex, they do not necessarily have closed form/simple solutions. On the other hand, simpler subproblems can be obtained via a specific DC decomposition described in \cite[Eq.~16]{PT1998} and many other related papers such as \cite{GTT2015}, i.e.,
\[
F(x) = \left(\frac{L}2\|x\|^2 + P_1(x)\right) - \left(\frac{L}2\|x\|^2-f(x)+P_2(x)\right),
\]
and we refer to the resulting DCA as the proximal DCA. When applied to solving \eqref{P1}, the subproblems of the proximal DCA take the following form:
\begin{equation}\label{pDCAsub}
\begin{split}
x^{t+1} &= \argmin_{x\in \R^n}\left\{\langle\nabla f(x^t)-\xi^t,x\rangle+\frac{L}2\|x-x^t\|^2 + P_1(x)\right\}\\
& = \argmin_{x\in \R^n}\left\{\frac{L}2\left\|x-\left(x^t - \frac1L[\nabla f(x^t)-\xi^t]\right)\right\|^2 + P_1(x)\right\},
\end{split}
\end{equation}
where $\xi^t\in \partial P_2(x^t)$, and $x^{t+1}$ is uniquely defined because $P_1$ is proper closed convex. In contrast to \eqref{DCAsub}, solving the subproblem \eqref{pDCAsub} amounts to evaluating the so-called proximal operator of $\frac1L P_1$, and this proximal operator is easy to compute for a wide variety of $P_1$; see, for example, \cite[Tables 10.1 and 10.2]{CP2011}.

Despite having simple subproblems for many commonly used $P_1$, the proximal DCA is potentially slow: this is because the proximal DCA is the same as the proximal gradient algorithm when $P_2=0$ and the proximal gradient algorithm can take a lot of iterations in practice \cite[Section~5]{DC2015}. Fortunately, the proximal gradient algorithm for convex problems (i.e., when $P_2=0$) has been successfully accelerated by various extrapolation techniques \cite{N1983,N2004,N2007,N20071}. Thus, it is tempting to incorporate extrapolation techniques into the proximal DCA to possibly accelerate the algorithm. Specifically, we consider the following algorithm for solving the DC optimization problem \eqref{P1}:

\begin{center}
\fbox{\parbox{5in}{\vspace{1mm}
~\textbf{Proximal difference-of-convex algorithm with extrapolation (${\bf pDCA}_e$)}:
\begin{description}
\item{\textbf{Input}:} $x^{0}\in {\rm dom}\, P_1$, $\{\beta_{t}\}\subseteq [0,1)$ with $\sup\limits_t \beta_t < 1$. Set $x^{-1}=x^{0}$.

\item \hspace{5mm} ~\textbf{for} $t=0,1,2,\cdots$ \\
\phantom{AAAAA;}Take any $\xi^t \in \partial P_2(x^t)$ and set\vspace{-1mm}
\begin{equation}\label{iteratePG}
\begin{split}
&y^{t}=x^{t}+\beta_{t}(x^{t}-x^{t-1}),\\
&x^{t+1}=\argmin_{y\in \R^n}\left\{\langle \nabla f(y^t) - \xi^t,y\rangle + \frac{L}2\|y - y^t\|^2 + P_1(y)\right\}.\\
\end{split}
\end{equation}
\item \vspace{0mm}\hspace{5mm} \textbf{end for}
\end{description}}}
\end{center}


In view of the algorithmic framework of ${\rm pDCA}_e$ and the subproblem \eqref{pDCAsub} in the proximal DCA, it is not hard to see that ${\rm pDCA}_e$ reduces to the proximal DCA when $\beta_t \equiv 0$. Hence, the proximal DCA is a special case of ${\rm pDCA}_e$. In addition, we would like to point out that the conditions on $\{\beta_t\}$ in ${\rm pDCA}_e$ (i.e., $\{\beta_{t}\}\subseteq [0,1)$ and $\sup\limits_t \beta_t < 1$) are general enough to cover many popular choices of extrapolation parameters including those used in FISTA with fixed restart or FISTA with both fixed and adaptive restart for solving \eqref{P1} with $P_2=0$ \cite{DC2015}.
In detail, in these schemes, one starts with $\theta_{-1}=\theta_0=1$, recursively defines for $t\ge 0$ that
\begin{equation}\label{fista_extra}
\beta_{t}=\frac{\theta_{t-1}-1}{\theta_{t}}\ \ {\rm with}\ \ \theta_{t+1}=\frac{1+\sqrt{1+4\theta_{t}^{2}}}{2},
\end{equation}
and resets $\theta_{t-1} = \theta_{t} = 1$ for some $t > 0$ under suitable conditions: in the fixed restart scheme, one fixes a positive number $\bar{T}$ and resets $\theta_{t-1} = \theta_{t} = 1$ every $\bar{T}$ iterations, while the adaptive restart scheme amounts to resetting $\theta_{t-1} = \theta_{t} = 1$ whenever $\langle y^{t-1}-x^{t},x^{t}-x^{t-1}\rangle >0$.
From these definitions, one can readily show by induction that the $\{\beta_t\}$ chosen as in FISTA with fixed restart or FISTA with both fixed and adaptive restart satisfies $\{\beta_{t}\}\subseteq [0,1)$ and $\sup\limits_t \beta_t < 1$.\footnote{Indeed, when $P_2 = 0$, FISTA with fixed restart and FISTA with both fixed and adaptive restart are special cases of ${\rm pDCA}_e$.} The choice of $\{\beta_t\}$ as in FISTA with both fixed and adaptive restart will be used in our numerical experiments in Section~\ref{sec5}.

\section{Convergence analysis}\label{sec4}
In this section, we study the convergence behavior of ${\rm pDCA}_e$. We first establish the global subsequential convergence of ${\rm pDCA}_e$. Then, by making an additional differentiability assumption on $P_2$ and assuming that the Kurdyka-{\L}ojasiewicz property holds for an auxiliary function, we prove the global convergence of the whole sequence generated by ${\rm pDCA}_e$ and analyze the rate of convergence.

\subsection{Convergence analysis I: Global subsequential convergence of ${\rm pDCA}_e$}

We start with the following definition of stationary points; see, for example, \cite[Remark 1]{GZLHY2013}. It is routine to show that any local minimizer of $F$ is a stationary point of $F$; see \cite[Theorem 2(i)]{PT1997}.
\begin{definition}
  Let $F$ be given in \eqref{P1}. We say that $\bar x$ is a stationary point of $F$ if
  \begin{equation*}
   0 \in \nabla f(\bar{x}) + \partial P_1(\bar{x}) - \partial P_2(\bar{x}).
  \end{equation*}
  The set of all stationary points of $F$ is denoted by ${\cal X}$.
\end{definition}

We are now ready to prove a global subsequential convergence result for ${\rm pDCA}_e$ applied to solving \eqref{P1}. Recall that $F$ in \eqref{P1} is level-bounded, and the extrapolation parameters $\{\beta_t\}$ in ${\rm pDCA}_e$ satisfy $\sup\limits_t\beta_t < 1$ and $\{\beta_t\}\subseteq [0,1)$.

\begin{theorem}{\bf (Global subsequential convergence of ${\rm pDCA}_e$)}\label{thm1}
Let $\{x^t\}$ be a sequence generated by ${\rm pDCA}_e$ for solving \eqref{P1}. Then the following statements hold.
\begin{enumerate}[{\rm (i)}]
  \item The sequence $\{x^t\}$ is bounded.
  \item $\lim_{t\to \infty}\|x^{t+1} - x^t\|= 0$.
  \item Any accumulation point of $\{x^t\}$ is a stationary point of $F$.
\end{enumerate}
\end{theorem}

\begin{proof}
  First we prove (i). We note from \eqref{iteratePG} that $x^{t+1}$ is the global minimizer of a strongly convex function. Using this and comparing the objective values of this strongly convex function at $x^{t+1}$ and $x^t$, we see immediately that
  \begin{equation}\label{fineq}
  \begin{split}
    &\langle \nabla f(y^t) - \xi^t,x^{t+1}\rangle + \frac{L}2\|x^{t+1} - y^t\|^2 + P_1(x^{t+1})\\
    &\le \langle \nabla f(y^t) - \xi^t,x^t\rangle + \frac{L}2\|x^t - y^t\|^2 + P_1(x^t) - \frac{L}2\|x^{t+1} - x^t\|^2.
  \end{split}
  \end{equation}
  On the other hand, using the fact that $\nabla f$ is Lipschitz continuous with a modulus of $L > 0$, we have
  \begin{equation}\label{eq1}
    \begin{split}
      &f(x^{t+1}) + P(x^{t+1})  \le f(y^t) + \langle \nabla f(y^t),x^{t+1} - y^t\rangle + \frac{L}2\|x^{t+1} - y^t\|^2 + P(x^{t+1})\\
      & = f(y^t) + \langle \nabla f(y^t),x^{t+1} - y^t\rangle + \frac{L}2\|x^{t+1} - y^t\|^2 + P_1(x^{t+1}) - P_2(x^{t+1})\\
      & \le f(y^t) + \langle \nabla f(y^t),x^{t+1} - y^t\rangle + \frac{L}2\|x^{t+1} - y^t\|^2 + P_1(x^{t+1}) - P_2(x^t) - \langle \xi^t,x^{t+1} - x^t\rangle\\
      &\le f(y^t) + \langle \nabla f(y^t),x^{t} - y^t\rangle + \frac{L}2\|x^t - y^t\|^2 + P_1(x^t) - P_2(x^t) - \frac{L}2\|x^{t+1} - x^t\|^2\\
      &\le f(x^t) + P(x^t) + \frac{L}2\|x^t - y^t\|^2 - \frac{L}2\|x^{t+1} - x^t\|^2,
    \end{split}
  \end{equation}
  where the second inequality follows from the subgradient inequality and the fact that $\xi^t\in \partial P_2(x^t)$, the third inequality follows from \eqref{fineq}, while the last inequality follows from the convexity of $f$ and the definition of $P$. Now, invoking the definition of $y^t$, we obtain further from \eqref{eq1} that
  \[
  f(x^{t+1}) + P(x^{t+1}) \le f(x^t) + P(x^t) + \frac{L}{2}\beta_t^2\|x^t - x^{t-1}\|^2 - \frac{L}2\|x^{t+1} - x^t\|^2.
  \]
  Consequently, we have upon rearranging terms that
  \begin{equation}\label{rel1}
    \frac{L}{2}(1-\beta_t^2)\|x^t - x^{t-1}\|^2 \le \left[f(x^t) + P(x^t) + \frac{L}2\|x^t - x^{t-1}\|^2\right] - \left[f(x^{t+1}) + P(x^{t+1}) + \frac{L}2\|x^{t+1} - x^t\|^2\right].
  \end{equation}

  Since $\{\beta_t\}\subset [0,1)$, we deduce from \eqref{rel1} that the sequence $\{f(x^t) + P(x^t) + \frac{L}2\|x^t - x^{t-1}\|^2\}$ is nonincreasing. This together with the fact that $x^0 = x^{-1}$ gives
  \[
  f(x^t) + P(x^t) \le f(x^t) + P(x^t) + \frac{L}2\|x^t - x^{t-1}\|^2 \le f(x^0) + P(x^0)
  \]
  for all $t\ge 0$, which shows that $\{x^t\}$ is bounded, thanks to the level-boundedness of $f+P$. This proves (i).

  Next we prove (ii). Summing both sides of \eqref{rel1} from $t=0$ to $\infty$, we obtain that
  \[
  \begin{split}
  \frac{L}{2}\sum_{t=0}^\infty(1-\beta_t^2)\|x^t - x^{t-1}\|^2 &\le f(x^0) + P(x^0) - \liminf_{t\to \infty}\left[f(x^{t+1}) + P(x^{t+1}) + \frac{L}2\|x^{t+1} - x^t\|^2\right]\\
  & \le f(x^0) + P(x^0) - v < \infty.
  \end{split}
  \]
  Since $\sup\limits_t\beta_t < 1$, we deduce immediately from the above relation that $\lim\limits_{t\to \infty}\|x^{t+1} - x^t\|= 0$. This proves (ii).

  Finally, let $\bar{x}$ be an accumulation point of $\{x^t\}$ and let $\{x^{t_{i}}\}$ be a subsequence such that $\lim\limits_{i\to\infty}x^{t_{i}}= \bar{x}$. Then, from the first-order optimality condition of the subproblem \eqref{iteratePG}, we have
\begin{equation*}
-L(x^{t_{i}+1}-y^{t_{i}})\in \partial P_1(x^{t_{i}+1})+\nabla f(y^{t_{i}})-\xi^{t_i}.
\end{equation*}
Using this together with the fact that $y^{t_{i}}=x^{t_{i}}+\beta_{t_{i}}(x^{t_{i}}-x^{t_{i}-1})$, we obtain further that
\begin{equation}\label{Optimali1}
-L[(x^{t_{i}+1}-x^{t_{i}})-\beta_{t_{i}}(x^{t_{i}}-x^{t_{i}-1})]\in \partial P_1(x^{t_{i}+1})+\nabla f(y^{t_{i}})-\xi^{t_i}.
\end{equation}
In addition, note that the sequence $\{\xi^{t_i}\}$ is bounded due to the continuity and convexity of $P_2$ and the boundedness of $\{x^{t_i}\}$. Thus, by passing to a further subsequence if necessary, we may assume without loss of generality that $\lim\limits_{i\to \infty}\xi^{t_i}$ exists, which belongs to $\partial P_2(\bar x)$ due to the closedness of $\partial P_2$.
Using this and invoking $\|x^{t_{i}+1}-x^{t_{i}}\|\to 0$ from (ii) together with the closedness of $\partial P_1$ and the continuity of $\nabla f$, we have upon passing to the limit in \eqref{Optimali1} that
\begin{equation*}
0\in \partial P_1(\bar{x})+\nabla f(\bar{x})-\partial P_2(\bar{x}).
\end{equation*}
This completes the proof.
\end{proof}

We next study the behavior of $\{F(x^t)\}$ for a sequence $\{x^t\}$ generated by ${\rm pDCA}_e$. The result will subsequently be used in establishing global convergence of the whole sequence $\{x^t\}$ under additional assumptions in the next subsection.

\begin{proposition}\label{Fconstant}
Let $\{x^{t}\}$ be a sequence generated by ${\rm pDCA}_e$ for solving \eqref{P1}. Then the following statements hold.
\begin{enumerate}[{\rm (i)}]
  \item $\zeta:=\lim\limits_{t\to\infty}F(x^t)$ exists.
  \item $F\equiv \zeta$ on $\Omega$, where $\Omega$ is the set of accumulation points of $\{x^{t}\}$.
\end{enumerate}
\end{proposition}
\begin{proof}
Since $\{\beta_t\}\subseteq [0,1)$, we see immediately from \eqref{rel1} that the sequence $\{F(x^t)+ \frac{L}2\|x^t - x^{t-1}\|^2\}$ is nonincreasing. In addition, this sequence is also bounded below by $v$. Furthermore, we recall from Theorem~\ref{thm1}(ii) that $\|x^{t+1}-x^{t}\| \rightarrow 0$. The conclusion that $\zeta:=\lim\limits_{t\rightarrow\infty}F(x^{t})$ exists now follows immediately from the aforementioned facts. This proves (i).

Now we prove (ii). We first note from Theorem~\ref{thm1}(i) and (iii) that $\emptyset \neq \Omega \subseteq\mathcal{X}$. Take any $\hat{x}\in \Omega$. By the definition of accumulation point, there exists a convergent subsequence $\{x^{t_{i}}\}$ such that $\lim\limits_{i\rightarrow\infty}x^{t_{i}}= \hat{x}$. Since $x^{t_i}$ is the minimizer of the subproblem \eqref{iteratePG}, we see that
\begin{equation*}
P_1(x^{t_{i}})+\langle \nabla f(y^{t_{i}-1})-\xi^{t_{i}-1}, x^{t_{i}}\rangle +\frac{L}{2}\|x^{t_{i}}-y^{t_{i}-1}\|^{2} \leq P_1(\hat{x})+\langle \nabla f(y^{t_{i}-1})-\xi^{t_{i}-1}, \hat{x}\rangle + \frac{L}{2}\|\hat{x}-y^{t_{i}-1}\|^{2}.
\end{equation*}
Rearranging terms, we obtain further that
\begin{equation}\label{uppersup}
P_1(x^{t_{i}})+\langle \nabla f(y^{t_{i}-1})-\xi^{t_{i}-1}, x^{t_{i}}-\hat{x}\rangle +\frac{L}{2}\|x^{t_{i}}-y^{t_{i}-1}\|^{2} \leq P_1(\hat{x})+\frac{L}{2}\|\hat{x}-y^{t_{i}-1}\|^{2}.
\end{equation}
On the other hand, observe that
\begin{equation}\label{separate2}
\|\hat x-y^{t_{i}-1}\| = \|\hat x-x^{t_{i}}+x^{t_{i}}-y^{t_{i}-1}\|\leq \|\hat x-x^{t_{i}}\|+\|x^{t_{i}}-y^{t_{i}-1}\|
\end{equation}
and that
\begin{equation}\label{separate1}
\begin{split}
\|x^{t_{i}}-y^{t_{i}-1}\| &= \|x^{t_{i}}-x^{t_{i}-1}-\beta_{t_{i}-1}(x^{t_{i}-1}-x^{t_{i}-2})\|\\
&\leq \|x^{t_{i}}-x^{t_{i}-1}\|+\|x^{t_{i}-1}-x^{t_{i}-2}\|,
\end{split}
\end{equation}
where we made use of the fact that $y^{t_{i}-1}=x^{t_{i}-1}+\beta_{t_{i}-1}(x^{t_{i}-1}-x^{t_{i}-2})$ for the equality.
Since $\|x^{t+1}-x^{t}\|\rightarrow 0$ from Theorem~\ref{thm1}(ii) and $\lim\limits_{i\rightarrow\infty}x^{t_{i}}= \hat{x}$, we have by passing to the limits in \eqref{separate2} and \eqref{separate1} that
\begin{equation}\label{separate3}
\|\hat x-y^{t_{i}-1}\|\to 0\ \ {\rm and}\ \
\|x^{t_{i}}-y^{t_{i}-1}\| \rightarrow 0.
\end{equation}
In addition, notice that the sequence $\{\xi^{t_i}\}$ is bounded, thanks to the convexity and continuity of $P_2$ and the fact that $\lim\limits_{i\rightarrow\infty}x^{t_{i}}= \hat{x}$. Using this and \eqref{separate3}, we obtain further that
\[
\begin{split}
  \zeta &=\lim\limits_{i\rightarrow\infty} f(x^{t_{i}})+P(x^{t_{i}})\\
  &=\lim\limits_{i\to \infty} f(x^{t_{i}}) + P(x^{t_{i}})+\langle \nabla f(y^{t_{i}-1})-\xi^{t_{i}-1}, x^{t_{i}}-\hat{x}\rangle +\frac{L}{2}\|x^{t_{i}}-y^{t_{i}-1}\|^{2} \\
  &\leq \limsup\limits_{i\to \infty} f(x^{t_{i}})+P_1(\hat{x})-P_2(x^{t_{i}})+\frac{L}{2}\|\hat{x}-y^{t_{i}-1}\|^{2} = F(\hat x),
\end{split}
\]
where the inequality follows from \eqref{uppersup} and the definition of $P$.
Finally, since $F$ is lower semicontinuous, we also have
\begin{equation*}
\begin{split}
F(\hat{x}) &\leq \liminf\limits_{i\rightarrow\infty}F(x^{t_{i}})=\lim\limits_{i\rightarrow\infty}F(x^{t_{i}})=\zeta.
\end{split}
\end{equation*}
Consequently, $F(\hat{x})=\lim\limits_{i\rightarrow\infty} F(x^{t_{i}})= \zeta$. Since $\hat{x}\in \Omega$ is arbitrary, we conclude that $F\equiv \zeta$ on $\Omega$. This completes the proof.
\end{proof}

\subsection{Convergence analysis II: Global convergence and convergence rate of the ${\rm pDCA}_e$}

In this subsection, we consider the global convergence property of the whole sequence $\{x^t\}$ generated by ${\rm pDCA}_e$ for solving \eqref{P1} and establish the convergence rate of $\{x^t\}$ under suitable conditions. We start by introducing the following assumption.
\begin{assumption}\label{assum1}
  The function $P_2$ in \eqref{P1} is continuously differentiable on an open set $\mathcal{N}_0$ that contains $\cal X$. Moreover, the gradient $\nabla P_2$ is locally Lipschitz continuous on ${\cal N}_0$.
\end{assumption}

While Assumption~\ref{assum1} may look restrictive at first glance, it is satisfied by many DC regularizers $P(x)$ that arise in applications. We present some concrete examples below.

\begin{example}\label{lsl1l2}
We consider the least squares problem with $\ell_{1-2}$ regularization \cite{YLHX2015}, which takes the following form
\begin{equation}\label{l1l2}
\min_{x\in \R^n} F_{\ell_{1-2}}(x) = \ \frac{1}{2}\|Ax-b\|^2 + \lambda\|x\|_1 - \lambda\|x\|,
\end{equation}
where $A\in {\mathbb{R}}^{m\times n}$, $b\in {\mathbb{R}}^{m}$ and $\lambda > 0$. We also assume that $A$ does not have zero columns so that $F_{\ell_{1-2}}$ is level-bounded (see \cite[Lemma~3.1]{YLHX2015} and \cite[Example~4.1(b)]{LiuP2016}). This model corresponds to \eqref{P1} with $f(x)=\frac{1}{2}\|Ax-b\|^2$, $P_1(x)= \lambda\|x\|_1$ and $P_2(x)=\lambda\|x\|$.

We claim that if $2\lambda < \|A^{T}b\|_{\infty}$, then $0$ is not a stationary point of $F_{\ell_{1-2}}$. Suppose to the contrary that $0\in \mathcal{X}$, then we have from the definition of stationary point that $A^{T}b \in \lambda\partial\|0\|_1 - \lambda\partial\|0\|$, which is equivalent to
\begin{equation*}
A^{T}b \in \lambda[-1,1]^n - \lambda B(0,1),
\end{equation*}
where $B(0,1)=\{x\in \R^n: \|x\|\le 1\}$. From this, we see that $\|A^{T}b\|_{\infty} \le 2\lambda$, which is a contradiction.

Hence, if $\lambda < \frac12\|A^{T}b\|_{\infty}$, then $\cal X$ does not contain $0$. Since $\cal X$ is closed, one can then construct an open set ${\cal N}_0$ containing $\cal X$ so that $P_2$ is continuously differentiable with locally Lipschitz gradient on ${\cal N}_0$. Thus, Assumption~\ref{assum1} is satisfied for \eqref{l1l2} when $\lambda < \frac12\|A^{T}b\|_{\infty}$.
\end{example}

\begin{example}\label{example:mcp}
We consider the minmax concave penalty (MCP) regularization \cite{Z2010}, whose DC decomposition is given in \cite{GZLHY2013}:
\[
P(x) = \lambda\sum_{i=1}^{n}\int_{0}^{|x_i|}\left[1-\frac{x}{\theta\lambda}\right]_{+}dx
= \lambda\|x\|_1 - \underbrace{\lambda\sum_{i=1}^{n} \int_{0}^{|x_i|}\min\left\{1,\frac{x}{\theta\lambda}\right\}dx}_{P_2(x)},
\]
where $\theta>0$ is a constant, $\lambda>0$ is the regularization parameter and $[x]_{+} = \max\{0,x\}$. It is routine to show that $P_2$ is continuously differentiable and
\begin{equation*}
\nabla_i P_2(x) = \lambda\,\mathrm{sign}(x_i)\min\{1,|x_i|/(\theta\lambda)\}.
\end{equation*}
Moreover, the gradient $\nabla P_2$ is Lipschitz continuous with modulus $\frac1\theta$.
\end{example}

\begin{example}\label{example3}
We consider the smoothly clipped absolute deviation (SCAD) regularization \cite{FL2001}, whose DC decomposition is given in \cite{GZLHY2013}:
\[
P(x) = \lambda\sum_{i=1}^{n}\int_{0}^{|x_i|}\min\left\{1, \frac{[\theta\lambda - x]_{+}}{(\theta-1)\lambda}\right\}dx = \lambda\|x\|_1 - \underbrace{\lambda\sum_{i=1}^{n} \int_{0}^{|x_i|}\frac{[\min\{\theta\lambda,x\}-\lambda]_+}{(\theta-1)\lambda}dx}_{P_2(x)},
\]
where $\lambda>0$ is the regularization parameter and $\theta>2$ is a constant. It is routine to show that $P_2$ is continuously differentiable with
\begin{equation*}
\nabla_i P_2(x) = \mathrm{sign}(x_i)\frac{[\min\{\theta\lambda,|x_i|\}-\lambda]_{+}}{\theta-1}.
\end{equation*}
Thus it is routine to show that $\frac1{\theta-1}$ is a Lipschitz continuity modulus of $\nabla P_2$.
\end{example}

\begin{example}\label{example4}
We consider the transformed $\ell_{1}$ regularization \cite{ZX2016}, whose DC decomposition is given in \cite{APX2016}:
\begin{equation*}
P(x) = \sum_{i=1}^{n}\frac{(a+1)|x_i|}{a+|x_i|} = \frac{a+1}{a}\|x\|_1 - \underbrace{ \sum_{i=1}^{n}\left[\frac{a+1}{a}|x_i|- \frac{(a+1)|x_i|}{a+|x_i|}\right]}_{P_2(x)},
\end{equation*}
where $a>0$. It was shown in \cite[Section~5.4]{APX2016} that $P_2(x)$ is continuously differentiable with a Lipschitz continuous gradient whose Lipschitz continuity modulus is $\frac{2(a+1)}{a^2}$.
\end{example}

\begin{example}\label{log}
The last regularization function we consider is the logarithmic penalty function \cite{CWB2008}, whose DC decomposition is given in \cite{GZLHY2013}:
\begin{equation*}
P(x) = \sum_{i=1}^{n}\left[\lambda\log (|x_i|+\epsilon) - \lambda\log \epsilon\right] = \frac{\lambda}{\epsilon}\|x\|_1 - \underbrace{ \sum_{i=1}^n \lambda\left[\frac{|x_i|}{\epsilon} - \log(|x_i|+\epsilon) + \log\epsilon\right]}_{P_2(x)},
\end{equation*}
where $\lambda$ and $\epsilon$ are positive numbers. One can see that $P_2(x)$ is continuously differentiable with a Lipschitz continuous gradient whose Lipschitz continuity modulus is $\frac{\lambda}{\epsilon^2}$.
\end{example}

We next present our global convergence analysis. We will show that the sequence $\{x^t\}$ generated by ${\rm pDCA}_e$ is convergent to a stationary point of $F$ under suitable assumptions. Our analysis follows a similar line of arguments to other convergence analysis based on KL property (see, for example, \cite{AB2009,ABRS2010,ABF2013,BB2016}), but has to make extensive use of the following auxiliary function:
\begin{equation}\label{Edef}
E(x, y)= f(x) + P(x) + \frac{L}2\|x - y\|^2.
\end{equation}

\begin{theorem}{\bf (Global convergence of ${\rm pDCA}_e$)}\label{seq_converge}
Suppose that Assumption~\ref{assum1} holds and $E$ is a KL function. Let $\{x^{t}\}$ be a sequence generated by ${\rm pDCA}_e$ for solving \eqref{P1}. Then the following statements hold.
\begin{enumerate}[{\rm (i)}]
\item $\lim\limits_{t\rightarrow \infty}\mathrm{dist}((0,0),\partial E(x^t, x^{t-1}))=0$.
\item The sequence $\{E(x^t,x^{t-1})\}$ is nonincreasing and $\lim\limits_{t\rightarrow \infty} E(x^t,x^{t-1})=\zeta$, where $\zeta$ is given in Proposition~\ref{Fconstant}.
\item The set of accumulation points of $\{(x^t,x^{t-1})\}$ is $\Upsilon := \{(x,x):\; x\in \Omega\}$ and $E\equiv \zeta$ on $\Upsilon$, where $\Omega$ is the set of accumulation points of $\{x^t\}$.
\item The sequence $\{x^t\}$ converges to a stationary point of $F$; moreover, $\sum_{t=1}^{\infty}\|x^t-x^{t-1}\|< \infty$.
\end{enumerate}
\end{theorem}

\begin{proof}
From Theorem \ref{thm1}(i), we see that $\{x^t\}$ is bounded. This together with the definition of $\Omega$ implies that $\lim\limits_{t\rightarrow \infty} \mathrm{dist}(x^t, \Omega)=0$. Also recall from Theorem \ref{thm1}(iii) that $\Omega \subseteq \mathcal{X}$. Thus, for any $\nu > 0$, there exists $T_0 > 0$ so that ${\rm dist}(x^t,\Omega) < \nu$ and $x^t \in \mathcal{N}_0$ whenever $t\ge T_0$, where $\mathcal{N}_0$ is the open set from Assumption~\ref{assum1}. Moreover, since $\Omega$ is compact due to the boundedness of $\{x^t\}$, by shrinking $\nu$ if necessary, we may assume without loss of generality that $\nabla P_2$ is globally Lipschitz continuous on the bounded set ${\cal N} :=\{x\in {\cal N}_0:\; {\rm dist}(x,\Omega) < \nu\}$.

Next, considering the subdifferential of the function $E$ in \eqref{Edef} at the point $(x^t, x^{t-1})$ for $t\ge T_0$, we have
\begin{equation}\label{subdiffx}
\partial E(x^t, x^{t-1}) = [\{\nabla f(x^t) - \nabla P_2(x^t) + L(x^t-x^{t-1})\} + \partial P_1(x^t)]\times \{-L(x^t-x^{t-1})\},
\end{equation}
where we made use of the definition of $P$, the facts that $P_2$ is continuously differentiable in $\cal N$ and that $x^t\in \cal N$ for $t\ge T_0$.

On the other hand, using the first-order optimality condition of the subproblem \eqref{iteratePG} in ${\rm pDCA}_e$, we have for any $t\ge T_0 + 1$ that
\begin{equation*}
-L(x^{t}-y^{t-1})-\nabla f(y^{t-1})+\nabla P_2(x^{t-1})\in \partial P_1(x^{t}),
\end{equation*}
since $P_2$ is continuously differentiable in $\cal N$ and $x^{t-1}\in \cal N$ whenever $t\ge T_0+1$.
Using this relation, we see further that
\begin{equation*}
\begin{split}
&-L(x^{t-1}-y^{t-1})+\nabla f(x^t)-\nabla f(y^{t-1})+\nabla P_2(x^{t-1})- \nabla P_2(x^t)\\
& = \nabla f(x^t)-\nabla P_2(x^{t})+ L(x^t-x^{t-1})-L(x^{t}-y^{t-1})-\nabla f(y^{t-1})+\nabla P_2(x^{t-1}) \\
& \in \nabla f(x^t)-\nabla P_2(x^{t})+ L(x^t-x^{t-1})+\partial P_1(x^t).
\end{split}
\end{equation*}
Combining this with \eqref{subdiffx}, we obtain
\[
(-L(x^{t-1}-y^{t-1})+\nabla f(x^t)-\nabla f(y^{t-1})+\nabla P_2(x^{t-1})- \nabla P_2(x^t),-L(x^t-x^{t-1}))\in \partial E(x^t,x^{t-1}).
\]
Using this, the definition of $y^t$ and the global Lipschitz continuity of $\nabla f$ and $\nabla P_2$ on $\cal N$, we see that there exists $C > 0$ such that
\begin{equation}\label{subgradientE}
\mathrm{dist}((0,0),\partial E(x^t, x^{t-1}))\le C(\|x^t-x^{t-1}\|+\|x^{t-1}-x^{t-2}\|)
\end{equation}
whenever $t\ge T_0 + 1$.
Since $\|x^{t+1}-x^{t}\|\rightarrow 0$ according to Theorem~\ref{thm1}(ii), we conclude that
\begin{equation*}
\lim\limits_{t\rightarrow \infty}\mathrm{dist}((0,0),\partial E(x^t, x^{t-1}))=0,
\end{equation*}
which proves (i).

We now prove (ii) and (iii). Using the fact that $\sup\limits_t\beta_t < 1$, the definition of $E$ and \eqref{rel1}, we see that there exists a positive number $D$ such that
\begin{equation}\label{decreaseE}
E(x^{t},x^{t-1})-E(x^{t+1},x^{t})\ge D\|x^t - x^{t-1}\|^2
\end{equation}
for all $t$.
In particular, the sequence $\{E(x^t, x^{t-1})\}$ is nonincreasing. Since this sequence is also bounded below by $v$, it is convergent. Next, in view of Theorem~\ref{thm1}(ii) which says that $\|x^t-x^{t-1}\|\rightarrow 0$, it is not hard to show that the set of accumulation points of $\{(x^t, x^{t-1})\}_{t\ge 1}$ is $\Upsilon$. Moreover,
\begin{equation*}
\lim_{t\to\infty}E(x^t,x^{t-1}) = \zeta,
\end{equation*}
thanks to Proposition~\ref{Fconstant}(i).
Furthermore, for any $(\hat x,\hat x)\in \Upsilon$ so that $\hat x\in \Omega$, we have $E(\hat x,\hat x) = F(\hat x)=\zeta$, where the last equality follows from Proposition~\ref{Fconstant}(ii). Since $\hat x\in \Omega$ is arbitrary, we conclude that $E\equiv \zeta$ on $\Upsilon$.  This proves (ii) and (iii).

Finally, we prove (iv). In view of Theorem~\ref{thm1}(iii), it suffices to show that $\{x^t\}$ is convergent.
We first consider the case that there exists a $t>0$ such that $E(x^{t},x^{t-1})=\zeta$. Since $\{E(x^t, x^{t-1})\}$ is nonincreasing and convergent to $\zeta$ due to (ii), we conclude that for any $\bar{t}\ge 0$,  $E(x^{t+\bar{t}},x^{t+\bar{t}-1})=\zeta$. Hence, we have from \eqref{decreaseE} that $x^t=x^{t+\bar{t}}$ for any $\bar{t}\ge 0$, meaning that $\{x^t\}$ converges finitely.

We next consider the case that $E(x^{t},x^{t-1})>\zeta$ for all $t$.
Since $E$ is a KL function, $\Upsilon$ is a compact subset of ${\rm dom}\, \partial E$ and $E\equiv \zeta$ on $\Upsilon$, by Lemma \ref{UNKLF}, there exist an $\epsilon>0$ and a continuous concave function $\phi\in \Xi_a$ with $a>0$ such that
\begin{equation}\label{loj}
\phi'(E(x,y)-\zeta)\mathrm{dist}((0,0),\partial E(x,y))\ge 1
\end{equation}
for all $(x,y) \in U$, where
\begin{equation*}
U = \left\{(x,y)\in \R^n\times\R^n: \mathrm{dist}((x,y),\Upsilon)<\epsilon \right\}\cap \left\{(x,y)\in \R^n\times\R^n: \zeta<E(x,y)<\zeta+a\right\}.
\end{equation*}
Since $\Upsilon$ is the set of accumulation points of $\{(x^t,x^{t-1})\}_{t\ge 1}$ by (iii), and $\{x^t\}$ is bounded due to Theorem~\ref{thm1}(i), we have
\[
\lim_{t\to \infty}{\rm dist}((x^t,x^{t-1}),\Upsilon) = 0.
\]
Hence, there exists $T_1>0$ such that $\mathrm{dist}((x^t,x^{t-1}),\Upsilon)<\epsilon$ whenever $t\ge T_1$. In addition, since the sequence $\{E(x^t, x^{t-1})\}$ is nonincreasing and convergent to $\zeta$ by (ii), there exists $T_2>0$ such that $\xi<E(x^t,x^{t-1})<\xi+a$ for all $t\ge T_2$. Taking $\bar{T} = \max\{T_0+1,T_1, T_2\}$, then the sequence $\{(x^t,x^{t-1})\}_{t\ge \bar{T}}$ belongs to $U$. Hence we deduce from \eqref{loj} that
\begin{equation}\label{lojseq}
\phi'(E(x^t,x^{t-1})-\zeta)\cdot \mathrm{dist}((0,0),\partial E(x^t,x^{t-1}))\ge 1,~~~\mbox{for all } t\ge \bar{T}.
\end{equation}
From the concavity of $\phi$, we see further that for any $t\ge \bar{T}$,
\begin{equation*}
\begin{split}
&\left[\phi(E(x^t, x^{t-1})-\zeta)-\phi(E(x^{t+1}, x^t)-\zeta)\right]\cdot\mathrm{dist}((0,0),\partial E(x^t,x^{t-1}))\\
&\ge \phi'(E(x^t, x^{t-1})-\zeta))\cdot\mathrm{dist}((0,0),\partial E(x^t,x^{t-1}))\cdot(E(x^t, x^{t-1})-E(x^{t+1}, x^t))\\
&\ge E(x^t, x^{t-1})-E(x^{t+1}, x^t),
\end{split}
\end{equation*}
where the last inequality holds due to \eqref{lojseq} and the fact that $\{E(x^t, x^{t-1})\}$ is nonincreasing.
Combining this with \eqref{subgradientE} and \eqref{decreaseE} and rearranging terms, we obtain that for any $t\ge \bar T$,
\begin{equation}\label{LOJ3}
\|x^t - x^{t-1}\|^2\le \frac{C}{D}\left(\phi(E(x^t, x^{t-1})-\zeta)-\phi(E(x^{t+1}, x^t)-\zeta)\right)\cdot \left(\|x^t-x^{t-1}\|+\|x^{t-1}-x^{t-2}\|\right).
\end{equation}
Taking square root on both sides of \eqref{LOJ3} and using the AM-GM inequality, we have
\begin{equation*}
\begin{split}
\|x^t - x^{t-1}\| &\le \sqrt{\frac{2C}{D}\left(\phi(E(x^t, x^{t-1})-\zeta)-\phi(E(x^{t+1}, x^t)-\zeta)\right)} \cdot \sqrt{\frac{\|x^t-x^{t-1}\|+\|x^{t-1}-x^{t-2}\|}{2}}\\
&\le \frac{C}{D}\left(\phi(E(x^t, x^{t-1})-\zeta)-\phi(E(x^{t+1}, x^t)-\zeta)\right) + \frac{1}{4}\|x^t-x^{t-1}\|+\frac{1}{4}\|x^{t-1}-x^{t-2}\|,
\end{split}
\end{equation*}
which implies that
\begin{equation}\label{lojseq-5}
\frac12\|x^t - x^{t-1}\| \le \frac{C}{D}\left(\phi(E(x^t, x^{t-1})-\zeta)-\phi(E(x^{t+1}, x^t)-\zeta)\right) + \frac14(\|x^{t-1}-x^{t-2}\|-\|x^{t}-x^{t-1}\|).
\end{equation}
Summing the above relation from $t=\bar T$ to $\infty$, we have
\begin{equation*}
\sum_{t=\bar T}^{\infty}\|x^t-x^{t-1}\|\le \frac{2C}{D}\phi(E(x^{\bar T}, x^{\bar T-1})-\zeta) + \frac12\|x^{\bar T-1}-x^{\bar T-2}\|< \infty,
\end{equation*}
which implies the convergence of $\{x^t\}$ as well as the summability of $\{\|x^{t+1}-x^t\|\}_{t\ge 0}$. This completes the proof.
\end{proof}
\begin{remark}
If the objective is not level bounded but we still have $v>-\infty$ (which can be true for least squares with regularizers in Examples~\ref{example:mcp}, \ref{example3} and \ref{example4}), we can still show that $\|x^t-x^{t-1}\| \rightarrow 0$ by following the same arguments as in the proof of Theorem~\ref{thm1}(ii). Consequently, if the sequence $\{x^t\}$ also has an accumulation point, then using a similar proof as Theorem~\ref{thm1}(iii), this accumulation point can be shown to be a stationary point of \eqref{P1}.
\end{remark}
We next consider the convergence rate of the sequence $\{x^t\}$ under the assumption that the auxiliary function $E$ is a KL function whose $\phi \in \Xi_a$ (see Definition~\ref{KLF}) takes the form $\phi(s)=cs^{1-\theta}$ for some $\theta\in [0,1)$. This kind of convergence rate analysis has also been performed for other optimization algorithms; see, for example, \cite{AB2009}. Our analysis is similar to theirs but makes use of the auxiliary function $E$ in \eqref{Edef}.

\begin{theorem}
Suppose that Assumption~\ref{assum1} holds. Let $\{x^{t}\}$ be a sequence generated by ${\rm pDCA}_e$ for solving \eqref{P1} and suppose that $\{x^{t}\}$ converges to some $\bar{x}$. Suppose further that $E$ is a KL function with $\phi$ in the KL inequality \eqref{KL_equation} taking the form $\phi(s)=cs^{1-\theta}$ for some $\theta \in [0,1)$ and $c>0$. Then the following statements hold.
\begin{enumerate}[{\rm (i)}]
  \item If $\theta=0$, then there exists $t_0>0$ so that $x^t$ is constant for $t> t_0$;
  \item If $\theta \in (0,\frac{1}{2}]$, then there exist $c_1>0$, $t_1>0$ and $\eta \in (0,1)$ such that $\|x^t-\bar{x}\| < c_1\eta^t$ for $t> t_1$;
  \item If $\theta \in (\frac{1}{2},1)$, then there exist $c_2>0$ and $t_2>0$ such that $\|x^t-\bar x\| < c_2 t^{-\frac{1-\theta}{2\theta-1}}$ for $t> t_2$.
\end{enumerate}
\end{theorem}
\begin{proof}
First, we prove (i). If $\theta=0$, we claim that there must exist $t_0>0$ such that $E(x^{t_0},x^{t_0-1})=\zeta$. Suppose to the contrary that $E(x^t,x^{t-1})>\zeta$ for all $t>0$. Since $\lim\limits_{t\to\infty} x^t = \bar x$ and the sequence $\{E(x^t, x^{t-1})\}$ is nonincreasing and convergent to $\zeta$ by Theorem~\ref{seq_converge}(ii), we have from $\phi(s)=cs$ and the KL inequality \eqref{lojseq} that for all sufficiently large $t$,
\begin{equation*}
\mathrm{dist}((0,0),\partial E(x^t,x^{t-1}))\ge\frac{1}{c},
\end{equation*}
which contradicts Theorem~\ref{seq_converge}(i). Thus, there exists $t_0>0$ so that $E(x^{t_0},x^{t_0-1})=\zeta$.
Since $\{E(x^t, x^{t-1})\}$ is nonincreasing and convergent to $\zeta$, it must then hold that $E(x^{t_0+\bar{t}},x^{t_0+\bar{t}-1})=\zeta$ for any $\bar t \ge 0$. Thus, we conclude from \eqref{decreaseE} that $x^{t_0}=x^{t_0+\bar{t}}$ for any $\bar{t}\ge 0$. This proves (i).

We next turn to the case that $\theta\in (0,1)$. If there exists $t_0>0$ such that $E(x^{t_0},x^{t_0-1})=\zeta$, then one can show that $\{x^t\}$ is finitely convergent as above, and the desired conclusions hold trivially. Hence, for $\theta \in (0,1)$, we only need to consider the case when $E(x^t,x^{t-1})>\zeta$ for all $t>0$.

Define $H_t = E(x^t, x^{t-1})-\zeta$ and $S_t = \sum_{i=t}^\infty \|x^{i+1}-x^{i}\|$, where $S_t$ is well-defined due to Theorem~\ref{seq_converge}(iv). Then, using \eqref{lojseq-5}, we have for any $t\ge \bar T$ (where $\bar T$ is defined as in \eqref{lojseq}) that
\begin{equation*}
\begin{split}
&S_{t} = 2\sum_{i=t}^\infty \frac12\|x^{i+1}-x^{i}\|\le 2\sum_{i=t}^\infty \frac12\|x^{i}-x^{i-1}\| \\
& \le 2\sum_{i=t}^\infty \left[\frac{C}{D}\left(\phi(E(x^i, x^{i-1})-\zeta)-\phi(E(x^{i+1}, x^i)-\zeta)\right) + \frac14(\|x^{i-1}-x^{i-2}\|-\|x^{i}-x^{i-1}\|)\right]\\
& \le \frac{2C}{D}\phi(E(x^t, x^{t-1})-\zeta) + \frac12\|x^{t-1}-x^{t-2}\| = \frac{2C}{D}\phi(H_t) + \frac12(S_{t-2} - S_{t-1}).
\end{split}
\end{equation*}
Using this and the fact that $\{S_t\}$ is nonincreasing, we obtain further that
\begin{equation}\label{rate2}
S_{t} \le \frac{2C}{D}\phi(H_t) + \frac12(S_{t-2}-S_{t})
\end{equation}
for all $t\ge \bar T$.
On the other hand, since $\lim\limits_{t\to\infty} x^t = \bar x$ and the sequence $\{E(x^t, x^{t-1})\}$ is nonincreasing and convergent to $\zeta$ by Theorem~\ref{seq_converge}(ii), we have from the KL inequality \eqref{lojseq} with $\phi(s)=cs^{1-\theta}$ that for all sufficiently large $t$,
\begin{equation}\label{rate3}
c(1-\theta)(H_t)^{-\theta}\mathrm{dist}((0,0),\partial E(x^t,x^{t-1})) \ge  1.
\end{equation}
In addition, using \eqref{subgradientE} and the definition of $S_t$, we see that for all sufficiently large $t$,
\begin{equation}\label{rate4}
\mathrm{dist}((0,0),\partial E(x^t,x^{t-1})) \le  C(S_{t-2}-S_{t}).
\end{equation}
Combining \eqref{rate3} and \eqref{rate4}, we have for all sufficiently large $t$ that
\begin{equation*}
(H_t)^{\theta} \le C\cdot c(1-\theta) \cdot (S_{t-2}-S_{t}).
\end{equation*}
Raising to a power of $\frac{1-\theta}{\theta}$ to both sides of the above inequality and scaling both sides by $c$, we obtain that
\begin{equation*}
c(H_t)^{1-\theta} \le c\cdot\left(C\cdot c(1-\theta) \cdot (S_{t-2}-S_{t})\right)^{\frac{1-\theta}{\theta}}.
\end{equation*}
Combining this with \eqref{rate2} and recalling that $\phi(H_t) = c(H_t)^{1-\theta}$, we see that for all sufficiently large $t$,
\begin{equation}\label{rate7}
S_{t} \le C_1(S_{t-2}-S_{t})^{\frac{1-\theta}{\theta}} + \frac12(S_{t-2}-S_{t})\le C_1(S_{t-2}-S_{t})^{\frac{1-\theta}{\theta}} + S_{t-2}-S_{t},
\end{equation}
where $C_1 = \frac{2C}{D}c\cdot\left(C\cdot c(1-\theta)\right)^{\frac{1-\theta}{\theta}}$.

We now consider two cases: $\theta \in (0,\frac12]$ or $\theta\in (\frac12,1)$.

Suppose first that $\theta \in (0,\frac{1}{2}]$. Then $\frac{1-\theta}{\theta}\ge 1$. Since $\|x^{t+1}-x^{t}\|\rightarrow 0$ from Theorem~\ref{thm1}(ii), it holds that $S_{t-2}-S_{t}\rightarrow 0$. From these and \eqref{rate7}, we conclude that there exists $t_1>0$ so that for all $t\ge t_1$, we have
\begin{equation*}
S_{t} \le (C_1+1)(S_{t-2}-S_{t}),
\end{equation*}
which implies that $S_{t} \le \frac{C_1+1}{C_1+2} S_{t-2}$. Hence,
\[
\|x^t-\bar{x}\| \le \sum_{i=t}^\infty \|x^{i+1}-x^i\| = S_t \le S_{t_1-2}\left(\sqrt{\frac{C_1+1}{C_1+2}}\right)^{t-t_1+1}
\]
for all $t\ge t_1$. This proves (ii).

Finally, we consider the case that $\theta \in (\frac{1}{2},1)$. In this case, we have $\frac{1-\theta}{\theta} <1$. Combining this with \eqref{rate7} and the fact that $S_{t-2}-S_{t} \rightarrow 0$, we see that there exists $t_2 > 0$ such that for all $t\ge t_2$, we have
\begin{equation*}
\begin{split}
S_{t} &\le C_1(S_{t-2}-S_{t})^{\frac{1-\theta}{\theta}} + S_{t-2}-S_{t}\\
& \le C_1(S_{t-2}-S_{t})^{\frac{1-\theta}{\theta}} + (S_{t-2}-S_{t})^{\frac{1-\theta}{\theta}}\\
& = (C_1+1)(S_{t-2}-S_{t})^{\frac{1-\theta}{\theta}}.
\end{split}
\end{equation*}
Raising to a power of $\frac{\theta}{1-\theta}$ to both sides of the above inequality, we see further that,
\begin{equation*}
S_t^{\frac{\theta}{1-\theta}}\le C_2(S_{t-2}-S_t)
\end{equation*}
whenever $t\ge t_2$, where $C_2 = (C_1+1)^{\frac{\theta}{1-\theta}}$. Consider the sequence $\Delta_t:= S_{2t}$. Then for any $t\ge \lceil\frac{t_2}2\rceil$, we have
\begin{equation*}
\Delta_t^{\frac{\theta}{1-\theta}}\le C_2(\Delta_{t-1}-\Delta_t).
\end{equation*}
Proceeding as in the proof of \cite[Theorem~2]{AB2009} starting from \cite[Equation~(13)]{AB2009}, one can show similarly that for all sufficiently large $t$,
\[
\Delta_t \le C_3 t^{-\frac{1-\theta}{2\theta-1}}
\]
for some $C_3 > 0$; see the first equation on \cite[Page~15]{AB2009}. This implies that for all sufficiently large $t$, we have
\begin{equation*}
\|x^t - \bar x\|\le S_t \begin{cases}
  = \Delta_{\frac{t}2}\le 2^\rho C_3 t^{-\rho}\ \ &\mbox{if $t$ is even},\\
  \le S_{t-1} = \Delta_{\frac{t-1}2}\le 2^\rho C_3 (t-1)^{-\rho} \le 4^\rho C_3 t^{-\rho}\ \ &\mbox{if $t$ is odd and $t\ge 2$},
\end{cases}
\end{equation*}
where $\rho := \frac{1-\theta}{2\theta-1}$.
This completes the proof.
\end{proof}

\begin{remark}\label{rem1}
We recall that there are many concrete examples of functions $f$ satisfying the KL property at all points in ${\rm dom}\,\partial f$ with $\phi(s)=cs^{1-\theta}$ for some $\theta \in [0,1)$ and $c>0$. Indeed, all proper closed semialgebraic functions satisfy this property; see, for example, \cite[section 2]{BDL2007} and \cite[section 4.3]{ABRS2010}. We refer the readers to \cite{ABRS2010,LP2016} for more examples. In particular, one can show that if $f(x)=\frac12\|Ax-b\|^2$ for some matrix $A$ and vector $b$, $P$ is given as in any one of the five examples at the beginning of this subsection, then the function $E$ in \eqref{Edef} is a KL function with $\phi(s)=cs^{1-\theta}$ for some $\theta \in [0,1)$ and $c>0$.
\end{remark}

\section{Numerical experiments}\label{sec5}

In this section, we perform numerical experiments to illustrate the efficiency of our algorithm ${\rm pDCA}_e$ for solving problem \eqref{P1}.
All experiments are performed in Matlab 2015b on a 64-bit PC with an Intel(R) Core(TM) i7-4790 CPU (3.60GHz) and 32GB of RAM.

In our numerical tests, we focus on the following DC regularized least squares problem:
\begin{equation}\label{dc_problem}
\min_{x\in \R^n} \frac{1}{2}\|Ax-b\|^2 + P_1(x) - P_2(x),
\end{equation}
where $A\in {\mathbb{R}}^{m\times n}$, $b\in {\mathbb{R}}^{m}$, $P_1$ is a proper closed convex function and $P_2$ is a continuous convex function. We consider two different classes of regularizers: the $\ell_{1-2}$ regularizer discussed in Example~\ref{lsl1l2} and the logarithmic regularizer presented in Example~\ref{log}.
We compare three algorithms for solving \eqref{dc_problem} with these regularizers: our algorithm ${\rm pDCA}_e$, the proximal DCA ({\rm pDCA}) studied in various work such as \cite{PT1998} and  \cite{GTT2015}, and the GIST proposed in \cite{GZLHY2013}. We discuss the implementation details of these algorithms below.

${\bf pDCA}_e$. For this algorithm, we set $L = \lambda_{\max}(A^TA)$,\footnote{$\lambda_{\max}(A^TA)$ is computed via the MATLAB code {\sf lambda = norm(A*A');} when $m\le 2000$, and by {\sf opts.issym = 1;
lambda= eigs(A*A',1,'LM',opts);} otherwise.} choose the extrapolation parameters $\{\beta_t\}$ as in \eqref{fista_extra}, and perform both the fixed restart (with $\bar{T} = 200$) and the adaptive restart strategies as described in Section~\ref{sec3}. We initialize the algorithm at the origin and terminate it when
\begin{equation*}
\frac{\|x^t - x^{t-1}\|}{\max\{1, \|x^t\|\}} < 10^{-5}.
\end{equation*}

${\bf pDCA}$. This is a special case of ${\rm pDCA}_e$ with $\beta_t \equiv 0$. We set $L = \lambda_{\max}(A^TA)$, initialize the algorithm at the origin and terminate it when
\begin{equation*}
\frac{\|x^t - x^{t-1}\|}{\max\{1, \|x^t\|\}} < 10^{-5}.
\end{equation*}
In our experiments below, this algorithm turns out to be very slow, and so we also terminate this algorithm when the iteration number hits 5000.

${\bf GIST}$. This algorithm was proposed in \cite{GZLHY2013}, and is the same as the nonmonotone proximal gradient algorithm described in \cite{WNF2009}
(see also \cite[Appendix A, Algorithm 1]{CLP2016}) applied to $f(x) = \frac{1}{2}\|Ax - b\|^2$ and $P(x) = P_1(x) - P_2(x)$.
Following the notation in \cite[Appendix A, Algorithm 1]{CLP2016}, in our implementation, we set $c = 10^{-4}$, $\tau = 2$, $M =4$, $L_0^0 = 1$, and
\begin{equation*}
L_t^0 = \min \left\{\max \left\{\frac{\|A(x^t - x^{t-1})\|^2}{\|x^t -x^{t-1}\|^2}, 10^{-8}\right\}, 10^8 \right\}
\end{equation*}
for $t\ge 1$. We would like to point out that the subproblem in \cite[Appendix A, A.4]{CLP2016} now becomes
\[
\min\limits_{x\in \R^n}\left\{\langle A^T(Ax^t - b),x - x^t \rangle + \frac{L_t}{2}\|x-x^t\|^2 + P_1(x) - P_2(x)\right\},
\]
which has closed form solutions for the two regularizers used in our experiments below; see the appendices of \cite{GZLHY2013} and \cite{LiuP2016}.
We initialize this algorithm at the origin and terminate it when
\begin{equation*}
\frac{\|x^t - x^{t-1}\|}{\max\{1, \|x^t\|\}} < 10^{-5}.
\end{equation*}

In our numerical experiments below, we compare our algorithm $\rm {pDCA}_e$ with $\rm {pDCA}$ and $\rm {GIST}$ for solving \eqref{dc_problem} on random instances generated as follows. We first generate an $m\times n$ matrix $A$ with i.i.d. standard Gaussian entries, and then normalize this matrix so that the columns of $A$ have unit norms. A subset $T$ of size $s$ is then chosen uniformly at random from $\{1,2,3,\ldots, n\}$ and an $s$-sparse vector $y$ having i.i.d. standard Gaussian entries on $T$ is generated. Finally, we set $b = Ay + 0.01\cdot \hat{n}$, where $\hat{n} \in \R^m$ is a random vector with i.i.d. standard Gaussian entries.

We next present the DC models we use in our numerical tests and the numerical results.

\subsection{Least squares problems with $\ell_{1-2}$ regularizer}

In this subsection, we consider the $\ell_{1-2}$ regularized least squares problem:
\begin{equation}\label{l1-l2}
\mathop{\min}_{x \in \R^n} F_{\ell_{1-2}}(x) = \frac{1}{2}\|Ax - b\|^2 + \lambda \|x\|_1 - \lambda \|x\|,
\end{equation}
where $A \in \R^{m \times n}$, $b \in \R^m$, and $\lambda > 0$ is the regularization parameter. This problem takes the form of \eqref{dc_problem} with $P_1(x) = \lambda \|x\|_1$ and $P_2(x)= \lambda \|x\|$. We assume {\em in addition} that the $A$ in \eqref{l1-l2} does not have zero columns. Using this assumption, Example~\ref{lsl1l2}, Theorem~\ref{seq_converge} and Remark~\ref{rem1}, we see that $F_{\ell_{1-2}}$ is level-bounded, and that if we choose $\lambda < \frac12\|A^{T}b\|_{\infty}$, then the sequence $\{x^t\}$ generated by ${\rm pDCA}_e$ is globally convergent.

In our numerical experiments below, we consider $(m,n,s) = (720i, 2560i, 80i)$ for $i = 1,2, \ldots, 10$. For each triple $(m,n,s)$, we generate 30 instances randomly as described above. The computational results are presented in Tables \ref{table1} and \ref{table2}, which correspond to problem \eqref{l1-l2} with $\lambda = 5 \times 10^{-4}$ and $\lambda = 1 \times 10^{-3}$ respectively.\footnote{These $\lambda$ satisfy $\lambda < \frac12\|A^{T}b\|_{\infty}$ for all our random instances.} We report the time for computing $\lambda_{\max}(A^TA)$ ($\bf t_{\lambda_{\max}}$), the number of iterations (iter),\footnote{In the tables, ``max" means the number of iterations hits 5000.} CPU times in seconds (CPU time),\footnote{The CPU time reported for ${\rm pDCA}_e$ does not include the time for computing $\lambda_{\max}(A^TA)$.} and the function values at termination (fval), averaged over the 30 random instances. We can see that ${\rm pDCA}_e$ always outperforms ${\rm pDCA}$ and ${\rm GIST}$.

\begin{table}[h]
\small
\caption{Solving \eqref{l1-l2} on random instances, $\lambda = 5 \times 10^{-4}$}\label{table1}
\hspace*{-1 cm}
\begin{tabular}{|c|c|c||c||c|c|c||c|c|c||c|c|c|}  \hline
\multicolumn{3}{|c||}{problem size}&  & \multicolumn{3}{c||}{iter} & \multicolumn{3}{c||}{CPU time} & \multicolumn{3}{c|}{fval}
\\ 
$n$ & $m$ &$s$& $\bf t_{\lambda_{\max}}$&{${\rm GIST}$} &{${\rm pDCA}_{e}$}&{${\rm pDCA}$}&{${\rm GIST}$}&{${\rm pDCA}_{e}$}&{${\rm pDCA}$} & {${\rm GIST}$}& {${\rm pDCA}_{e}$}& {${\rm pDCA}$}
\\ \hline
  2560 &   720 &    80 &   0.1 &  1736 &   915 &  {\rm max} &   3.0 &   1.2 &   6.1 & 2.9757e-02 & 2.9743e-02 & 4.7049e-02  \\
  5120 &  1440 &   160 &   0.7 &  1726 &   895 &  {\rm max} &  14.2 &   5.4 &  29.6 & 6.1497e-02 & 6.1472e-02 & 9.5797e-02  \\
  7680 &  2160 &   240 &   0.7 &  1747 &   929 &  {\rm max} &  31.0 &  12.1 &  64.7 & 9.3836e-02 & 9.3799e-02 & 1.4394e-01  \\
 10240 &  2880 &   320 &   1.3 &  1754 &   949 &  {\rm max} &  54.8 &  21.8 & 114.6 & 1.2500e-01 & 1.2495e-01 & 1.9063e-01  \\
 12800 &  3600 &   400 &   2.4 &  1767 &   935 &  {\rm max} &  86.8 &  33.9 & 180.6 & 1.5956e-01 & 1.5949e-01 & 2.4367e-01  \\
 15360 &  4320 &   480 &   3.7 &  1757 &   955 &  {\rm max} & 120.5 &  48.5 & 253.2 & 1.8982e-01 & 1.8975e-01 & 2.8811e-01  \\
 17920 &  5040 &   560 &   6.0 &  1778 &   982 &  {\rm max} & 166.6 &  67.5 & 343.7 & 2.2481e-01 & 2.2472e-01 & 3.4110e-01  \\
 20480 &  5760 &   640 &   7.6 &  1780 &   982 &  {\rm max} & 215.6 &  87.5 & 444.4 & 2.5908e-01 & 2.5897e-01 & 3.9319e-01  \\
 23040 &  6480 &   720 &  10.7 &  1782 &   982 &  {\rm max} & 269.8 & 110.4 & 561.4 & 2.9150e-01 & 2.9137e-01 & 4.4057e-01  \\
 25600 &  7200 &   800 &  14.3 &  1799 &   995 &  {\rm max} & 341.4 & 140.1 & 704.0 & 3.2831e-01 & 3.2816e-01 & 4.9679e-01  \\ \hline
\end{tabular}
\normalsize
\end{table}

\begin{table}[h]
\small
\caption{Solving \eqref{l1-l2} on random instances, $\lambda = 1 \times 10^{-3}$}\label{table2}
\hspace*{-1 cm}
\begin{tabular}{|c|c|c||c||c|c|c||c|c|c||c|c|c|}  \hline
\multicolumn{3}{|c||}{problem size}&  & \multicolumn{3}{c||}{iter} & \multicolumn{3}{c||}{CPU time} &  \multicolumn{3}{c|}{fval}
\\ 
$n$ & $m$ &$s$& $\bf t_{\lambda_{\max}}$&{${\rm GIST}$} &{${\rm pDCA}_{e}$}&{${\rm pDCA}$}&{${\rm GIST}$}&{${\rm pDCA}_{e}$}&{${\rm pDCA}$} & {${\rm GIST}$}& {${\rm pDCA}_{e}$}& {${\rm pDCA}$}
\\ \hline
  2560 &   720 &    80 &   0.1 &   925 &   600 &  {\rm max} &   1.7 &   0.8 &   6.1 & 5.9909e-02 & 5.9903e-02 & 7.2646e-02  \\
  5120 &  1440 &   160 &   0.7 &   908 &   602 &  {\rm max} &   7.4 &   3.6 &  29.5 & 1.2002e-01 & 1.2001e-01 & 1.4286e-01  \\
  7680 &  2160 &   240 &   0.6 &   928 &   602 &  {\rm max} &  16.4 &   7.9 &  65.3 & 1.8679e-01 & 1.8677e-01 & 2.2359e-01  \\
 10240 &  2880 &   320 &   1.3 &   941 &   602 &  {\rm max} &  29.0 &  13.8 & 114.4 & 2.5185e-01 & 2.5182e-01 & 3.0125e-01  \\
 12800 &  3600 &   400 &   2.4 &   946 &   602 &  {\rm max} &  45.8 &  21.8 & 179.9 & 3.1906e-01 & 3.1903e-01 & 3.8187e-01  \\
 15360 &  4320 &   480 &   3.8 &   949 &   602 &  {\rm max} &  64.6 &  30.6 & 253.5 & 3.8418e-01 & 3.8414e-01 & 4.6012e-01  \\
 17920 &  5040 &   560 &   6.0 &   943 &   602 &  {\rm max} &  87.4 &  41.7 & 345.8 & 4.4659e-01 & 4.4654e-01 & 5.3129e-01  \\
 20480 &  5760 &   640 &   7.7 &   946 &   602 &  {\rm max} & 112.4 &  53.6 & 444.4 & 5.1037e-01 & 5.1031e-01 & 6.0884e-01  \\
 23040 &  6480 &   720 &  10.6 &   943 &   602 &  {\rm max} & 141.8 &  68.2 & 562.0 & 5.8029e-01 & 5.8022e-01 & 6.9129e-01  \\
 25600 &  7200 &   800 &  14.1 &   946 &   602 &  {\rm max} & 179.0 &  84.9 & 703.7 & 6.4830e-01 & 6.4822e-01 & 7.7247e-01  \\
\hline
\end{tabular}
\normalsize
\end{table}

\subsection{Least squares problems with logarithmic regularizer}
In this subsection, we consider the least squares problem with logarithmic regularization function:
\begin{equation}\label{ls-log}
\mathop{\min}_{x \in \R^n} F_{\rm log}(x) = \frac{1}{2}\|Ax - b\|^2 + \sum_{i=1}^{n}\left[\lambda\log (|x_i|+\epsilon) - \lambda\log \epsilon\right],
\end{equation}
where $A \in \R^{m \times n}$, $b \in \R^m$, $\epsilon >0$ is a constant, and $\lambda > 0$ is the regularization parameter. From the discussion in Example \ref{log}, it is easy to show that $F_{\rm log}$ takes the form of \eqref{dc_problem} with $P_1(x) = \frac{\lambda}{\epsilon}\|x\|_1$ and $P_2(x)= \sum_{i=1}^n \lambda\left[\frac{|x_i|}{\epsilon} - \log(|x_i|+\epsilon) + \log\epsilon\right]$. In addition, it is not hard to show that $F_{\rm log}$ is level-bounded. This together with Theorem \ref{seq_converge} and Remark~\ref{rem1} shows that the sequence $\{x^t\}$ generated by ${\rm pDCA}_e$ is globally convergent to a stationary point of \eqref{ls-log}.

In our experiments below, we consider $(m,n,s) = (720i, 2560i, 80i)$, $i = 1,2, \ldots, 10$. For each triple, we generate 30 instances randomly as described above. The computational results are presented in Tables \ref{t3} and \ref{t4}, which correspond to problem \eqref{ls-log} with $\lambda = 5 \times 10^{-4}$ and $\lambda = 1 \times 10^{-3}$ respectively.\footnote{We set $\epsilon = 0.5$ in \eqref{ls-log}.} In these tables, we report the time for computing $\lambda_{\max}(A^TA)$ ($\bf t_{\lambda_{\max}}$), the number of iterations (iter),\footnote{In the tables, ``max" means the number of iterations hits 5000.} CPU times in seconds (CPU time),\footnote{The CPU time reported for ${\rm pDCA}_e$ does not include the time for computing $\lambda_{\max}(A^TA)$.} and the function values at termination (fval), averaged over the 30 random instances. We see from the tables that ${\rm pDCA}_e$ always outperforms ${\rm pDCA}$ and ${\rm GIST}$.

\begin{table}[h]
\small
\caption{Solving \eqref{ls-log} on random instances, $\lambda = 5 \times 10^{-4}$} \label{t3}
\hspace*{-1 cm}
\begin{tabular}{|c|c|c||c||c|c|c||c|c|c||c|c|c|}  \hline
\multicolumn{3}{|c||}{problem size}&  & \multicolumn{3}{c||}{iter} & \multicolumn{3}{c||}{CPU time} &   \multicolumn{3}{c|}{fval}
\\ 
$n$ & $m$ &$s$& $\bf t_{\lambda_{\max}}$&{${\rm GIST}$} &{${\rm pDCA}_{e}$}&{${\rm pDCA}$}&{${\rm GIST}$}&{${\rm pDCA}_{e}$}&{${\rm pDCA}$} & {${\rm GIST}$}& {${\rm pDCA}_{e}$}& {${\rm pDCA}$}
\\ \hline
  2560 &   720 &    80 &   0.1 &  863 &   601 &  {\rm max} &   1.9 &   0.8 &   6.1 & 3.8020e-02 & 3.8013e-02 & 5.3479e-02  \\
  5120 &  1440 &   160 &   0.7 &  866 &   602 &  {\rm max} &   7.4 &   3.6 &  29.4 & 7.5865e-02 & 7.5852e-02 & 1.0691e-01  \\
  7680 &  2160 &   240 &   0.7 &  878 &   602 &  {\rm max} &  16.0 &   7.8 &  64.9 & 1.1419e-01 & 1.1417e-01 & 1.6253e-01  \\
 10240 &  2880 &   320 &   1.3 &  866 &   602 &  {\rm max} &  27.2 &  13.8 & 113.8 & 1.5219e-01 & 1.5217e-01 & 2.1442e-01  \\
 12800 &  3600 &   400 &   2.4 &  869 &   602 &  {\rm max} &  43.1 &  22.0 & 181.9 & 1.8917e-01 & 1.8914e-01 & 2.6717e-01  \\
 15360 &  4320 &   480 &   3.7 &  869 &   602 &  {\rm max} &  59.9 &  30.9 & 256.0 & 2.2823e-01 & 2.2819e-01 & 3.2213e-01  \\
 17920 &  5040 &   560 &   6.0 &  866 &   602 &  {\rm max} &  80.7 &  41.8 & 346.6 & 2.6594e-01 & 2.6589e-01 & 3.7583e-01  \\
 20480 &  5760 &   640 &   7.7 &  874 &   602 &  {\rm max} & 104.9 &  53.8 & 446.4 & 3.0510e-01 & 3.0505e-01 & 4.3300e-01  \\
 23040 &  6480 &   720 &  10.7 &  873 &   602 &  {\rm max} & 132.0 &  67.9 & 563.1 & 3.4211e-01 & 3.4205e-01 & 4.8604e-01  \\
 25600 &  7200 &   800 &  14.3 &  871 &   602 &  {\rm max} & 164.7 &  85.0 & 705.0 & 3.8055e-01 & 3.8049e-01 & 5.4107e-01  \\
\hline
\end{tabular}
\normalsize
\end{table}

\begin{table}[h]
\small
\caption{Solving \eqref{ls-log} on random instances, $\lambda = 1 \times 10^{-3}$}\label{t4}
\hspace*{-1 cm}
\begin{tabular}{|c|c|c||c||c|c|c||c|c|c||c|c|c|}  \hline
\multicolumn{3}{|c||}{problem size}&  & \multicolumn{3}{c||}{iter} & \multicolumn{3}{c||}{CPU time} &     \multicolumn{3}{c|}{fval}
\\ 
$n$ & $m$ &$s$& $\bf t_{\lambda_{\max}}$&{${\rm GIST}$} &{${\rm pDCA}_{e}$}&{${\rm pDCA}$}&{${\rm GIST}$}&{${\rm pDCA}_{e}$}&{${\rm pDCA}$} & {${\rm GIST}$}& {${\rm pDCA}_{e}$}& {${\rm pDCA}$}
\\ \hline
  2560 &   720 &    80 &   0.1 &  473 &   380 &  4531 &   1.0 &   0.5 &   5.7 & 7.6101e-02 & 7.6099e-02 & 7.6125e-02  \\
  5120 &  1440 &   160 &   0.7 &  473 &   400 &  4540 &   4.1 &   2.4 &  27.1 & 1.5200e-01 & 1.5200e-01 & 1.5204e-01  \\
  7680 &  2160 &   240 &   0.7 &  467 &   402 &  4546 &   8.4 &   5.3 &  59.7 & 2.2691e-01 & 2.2691e-01 & 2.2696e-01  \\
 10240 &  2880 &   320 &   1.3 &  475 &   402 &  4549 &  14.6 &   9.2 & 103.4 & 3.0374e-01 & 3.0373e-01 & 3.0381e-01  \\
 12800 &  3600 &   400 &   2.4 &  470 &   401 &  4519 &  22.7 &  14.5 & 162.7 & 3.7530e-01 & 3.7529e-01 & 3.7538e-01  \\
 15360 &  4320 &   480 &   3.8 &  471 &   402 &  4539 &  31.9 &  20.5 & 230.6 & 4.5451e-01 & 4.5450e-01 & 4.5461e-01  \\
 17920 &  5040 &   560 &   6.1 &  471 &   402 &  4564 &  42.9 &  27.6 & 312.5 & 5.2941e-01 & 5.2939e-01 & 5.2953e-01  \\
 20480 &  5760 &   640 &   7.8 &  475 &   402 &  4554 &  56.1 &  35.9 & 406.5 & 6.0388e-01 & 6.0385e-01 & 6.0401e-01  \\
 23040 &  6480 &   720 &  10.7 &  476 &   402 &  4593 &  70.9 &  45.3 & 516.9 & 6.8519e-01 & 6.8516e-01 & 6.8534e-01  \\
 25600 &  7200 &   800 &  14.3 &  475 &   402 &  4559 &  88.2 &  56.8 & 642.5 & 7.5684e-01 & 7.5681e-01 & 7.5701e-01  \\
\hline
\end{tabular}
\normalsize
\end{table}

\section{Conclusion }\vspace{0.5ex}

In this paper, we propose a proximal difference-of-convex algorithm with extrapolation ($\rm {pDCA}_e$) for solving \eqref{P1}, which reduces to the proximal DCA when $\beta_t \equiv 0$. Our algorithmic framework allows a wide range of choices of the extrapolation parameters $\{\beta_t\}$, including those used in FISTA with fixed restart \cite{DC2015}. We establish global subsequential convergence of the sequence generated by $\rm {pDCA}_e$. In addition, by assuming the Kurdyka-{\L}ojasiewicz property of the objective and the locally Lipschitz differentiability of $P_2(x)$ in \eqref{P1}, we establish global convergence of the sequence generated by our algorithm and analyze its convergence rate. Our numerical experiments show that our algorithm usually outperforms the proximal DCA and GIST for two classes of DC regularized least squares problems.

{\bf Acknowledgement}. The authors would like to thank the two anonymous referees for their helpful comments.


\begin{thebibliography}{00}
\bibitem{APX2016}
 M. Ahn, J.S. Pang, and J. Xin.
\newblock Difference-of-convex learning I: directional stationarity, optimality, and sparsity.
\newblock {\em preprint}, 2016.

\bibitem{ASP2014}
A. Alvarado, G. Scutari, and J.S. Pang.
\newblock  A new decomposition method for multiuser DC-programming and its applications.
\newblock {\em IEEE Transactions on Signal Processing}, 62: 2984--2998, 2014.

\bibitem{AB2009}
 H. Attouch and J. Bolte.
\newblock On the convergence of the proximal algorithm for nonsmooth functions invoving analytic features.
\newblock {\em Mathematical Programming, Series B}, 116: 5--16, 2009.

\bibitem{ABRS2010}
 H. Attouch, J. Bolte, P. Redont, and A. Soubeyran.
\newblock Proximal alternating minimization and projection methods for nonconvex problems: an approach based on the Kurdyka-${\L}$ojasiewicz inequality.
\newblock {\em Mathematics of Operations Research}, 35: 438--457, 2010.

\bibitem{ABF2013}
 H. Attouch, J. Bolte, and B. F. Svaiter.
\newblock Convergence of descent methods for semi-algebraic and tame problems: proximal algorithms, forward-backward splitting, and regularized Gauss-Seidel methods.
\newblock {\em Mathematical Programming, Series A}, 137: 91--129, 2013.

\bibitem{BB2016}
 S. Banert and R.I. Bo\c{t}.
\newblock A general double-proximal gradient algorithm for d.c. programming.
\newblock {\em arXiv preprint arXiv}:1610.06538v1.

\bibitem{BT2009}
A. Beck and M. Teboulle.
\newblock A fast iterative shrinkage-thresholding algorithm for linear inverse problems.
\newblock {\em SIAM Journal on Imaging Sciences}, 2: 183--202, 2009.

\bibitem{BCG11}
 S. Becker, E.J. Cand\`{e}s, and M.C. Grant.
\newblock Templates for convex cone problems with applications to sparse signal recovery.
\newblock {\em Mathematical Programming Computation}, 3: 165--218, 2011.

\bibitem{BC2016}
W. Bian and X. Chen.
\newblock Optimality and complexity for constrained optimization problems with nonconvex regularization.
\newblock To appear in {\em Mathematics of Operations Research}.


\bibitem{BDL2007}
 J. Bolte, A. Daniilidis, and A. Lewis.
\newblock The {\L}ojasiewicz inequality for nonsmooth subanalytic functions
with applications to subgradient dynamical systems.
\newblock {\em SIAM Journal on Optimization}, 17: 1205--1223, 2007.

\bibitem{BST2014}
 J. Bolte, S. Sabach, and M. Teboulle.
\newblock Proximal alternating linearized minimization for nonconvex and nonsmooth problems.
\newblock {\em Mathematical Programming, Series A}, 146: 459--494, 2014.


\bibitem{CWB2008}
E.J. Cand\`{e}s, M. Wakin, and S. Boyd.
\newblock Enhancing spasity by reweighted $\ell_{1}$ minimization.
\newblock {\em Journal of Fourier Analysis and Applications}, 14: 877--905, 2008.

\bibitem{CLP2016}
 X. Chen, Z. Lu, and T.K. Pong.
\newblock Penalty methods for a class of non-Lipschitz optimization problems.
\newblock {\em SIAM Journal on Optimization}, 26: 1465--1492, 2016.


\bibitem{CP2011}
P.L. Combettes and J.-C. Pesquet.
\newblock Proximal splitting methods in signal processing.
\newblock {\em Fixed-Point Algorithms for Inverse Problems in Science and Engineering}, 49: 185--212, 2011.

\bibitem{DC2015}
 B. O'Donoghue and E.J. Cand\`{e}s.
\newblock Adaptive restart for accelerated gradient schemes.
\newblock {\em Foundations of Computational Mathematics}, 15: 715--732, 2015.

\bibitem{FL2001}
 J. Fan and R. Li.
\newblock Variable selection via nonconcave penalized likelihood and its oracle properties.
\newblock {\em Journal of the American Statistical Association}, 96: 1348--1360, 2001.

\bibitem{GZLHY2013}
P. Gong, C. Zhang, Z. Lu, J. Z. Huang, and J. Ye.
\newblock A general iterative shinkage and thresholding algorithm for non-convex regularized optimization problems.
\newblock {\em ICML}, 2013.

\bibitem{GTT2015}
J. Gotoh, A. Takeda, and K. Tono.
\newblock DC formulations and algorithms for sparse optimization problems.
\newblock Preprint, {\em METR 2015-27}, Department of Mathematical Informatics, University of Tokyo. Available at \verb+http://www.keisu.t.u-tokyo.ac.jp/research/techrep/index.html+

\bibitem{LP2016}
 G. Li and T.K. Pong.
\newblock Calculus of the exponent of Kurdyka-${\L}$ojasiewicz inequality and its applications to linear convergence of first-order methods.
\newblock {\em arXiv preprint arXiv}:1602.02915v3.

\bibitem{LiuP2016}
 T. Liu and T.K. Pong.
\newblock Further properties of the forward-backward envelope with applications to difference-of-convex programming.
\newblock {\em Computational Optimization and Applications}, 67: 489--520, 2017.

\bibitem{N1983}
Y. Nesterov.
\newblock A method of solving a convex programming problem with convergence rate $O(\frac1{k^2})$.
\newblock {\em Soviet Mathematics Doklady}, 27: 372--376, 1983.

\bibitem{N2004}
Y. Nesterov.
\newblock {\em Introductory Lectures on Convex Optimization: A Basic Course.}
\newblock Kluwer Academic Publishers, Boston, 2004.

\bibitem{N2007}
Y. Nesterov.
\newblock Gradient methods for minimizing composite functions.
\newblock {\em Mathematical Programming, Series B}, 140: 125--161, 2013.

\bibitem{N20071}
Y. Nesterov.
\newblock Dual extrapolation and its applications to solving variational inequalities and related problems.
\newblock {\em Mathematical Programming, Series B}, 109: 319--344, 2007.

\bibitem{P1964}
B.T. Polyak.
\newblock Some methods of speeding up the convergence of iteration methods.
\newblock {\em USSR Computational Mathematics and Mathematical Physics}, 4: 1--17, 1964.

\bibitem{PT1997}
D.T. Pham and H.A. Le Thi.
\newblock Convex analysis approach to D.C. programming: theory, algorithms and applications.
\newblock {\em Acta Mathematica Vietnamica}, 22: 289--355, 1997.

\bibitem{PT1998}
D.T. Pham and H.A. Le Thi.
\newblock A D.C. optimization algorithm for solving the trust-region subproblem.
\newblock {\em SIAM Journal on Optimization}, 8: 476--505, 1998.

\bibitem{RR1998}
 R.T. Rockafellar and R.J.-B. Wets.
\newblock {\em Variational Analysis}.
\newblock Springer, 1998.

\bibitem{SRL2014}
M. Sanjabi, M. Razaviyayn, and Z.-Q. Luo.
\newblock  Optimal joint base station assignment and beamforming for heterogeneous networks.
\newblock {\em IEEE Transactions on Signal Processing}, 62: 1950--1961,2014.

\bibitem{TP2005}
H.A. Le Thi and D.T. Pham.
\newblock The DC (difference of convex functions) programming and DCA revisited with DC models of real world nonconvex optimization problems.
\newblock {\em Annals of Operations Research}, 133: 23--46, 2005.

\bibitem{TPH2012}
H.A. Le Thi, D.T. Pham, and V.N. Huynh.
\newblock Exact penalty and error bounds in DC programming.
\newblock {\em Journal of Global Optimization}, 52: 509--535, 2012.

\bibitem{TPL1999}
H.A. Le Thi, D.T. Pham, and D.M. Le.
\newblock Exact penalty in D.C. programming.
\newblock {\em Vietnam Journal of Mathematics}, 27: 169--178, 1999.

\bibitem{T2016}
H. Tuy.
\newblock {\em Convex Analysis and Global Optimization, Second Edition}.
\newblock Springer, 2016.

\bibitem{WNF2009}
S.J. Wright, R. Nowak, and M. A. T. Figueiredo.
\newblock Sparse reconstruction by separable approximation.
\newblock {\em IEEE Transactions on Signal Processing}, 57: 2479--2493, 2009.

\bibitem{YLHX2015}
 P. Yin, Y. Lou, Q. He, and J. Xin.
\newblock Minimization of $\ell_{1-2}$ for compressed sensing.
\newblock {\em SIAM Journal on Scientific Computing}, 37: A536--A563, 2015.

\bibitem{Z2010}
 C. Zhang.
\newblock Nearly unbiased variable selection under minimax concave penalty.
\newblock {\em Annals of Statistics}, 38: 894--942, 2010.

\bibitem{ZX2016}
 S. Zhang and J. Xin.
\newblock Minimization of transformed $L_1$ penalty: theory, difference of convex
function algorithm, and robust application in compressed sensing.
\newblock {\em arXiv preprint arXiv}:1411.5735v3.


\end{thebibliography}
\end{document}